\date{January, 2018}
\documentclass[12pt]{amsart}
\usepackage{latexsym,amsmath,amsfonts,amscd,amssymb}
\usepackage{graphics}
\newtheorem{theorem}{Theorem}
\newtheorem{theorem*}{Theorem*}
\newtheorem{lemma}[theorem]{Lemma}
\newtheorem{lemma*}[theorem*]{Lemma}
\newtheorem{proposition}[theorem]{Proposition}
\newtheorem{proposition*}[theorem*]{Proposition}
\newtheorem{corollary}[theorem]{Corollary}
\newtheorem{corollary*}[theorem*]{Corollary}

\newtheorem{definition*}[theorem*]{Definition}
\newtheorem{conjecture}[theorem]{Conjecture}

\newtheorem{problem}[theorem]{Problem}

\theoremstyle{remark}

\newcommand{\cD}{{\mathcal D}}

\newcommand{\cF}{{\mathcal F}}

\newcommand{\cO}{{\mathcal O}}

\newcommand{\CC}{{\mathbb C}}
\newcommand{\DD}{{\mathbb D}}

\newcommand{\QQ}{{\mathbb Q}}
\newcommand{\RR}{{\mathbb R}}
\newcommand{\TT}{{\mathbb T}}
\renewcommand{\SS}{{\mathbb S}}
\newcommand{\ZZ}{{\mathbb Z}}

\renewcommand{\a}{\alpha}

\newcommand{\eps}{\epsilon}

\renewcommand{\o}{\omega}

\newcommand{\D}{\Delta}

\title{Solution to Briot and Bouquet problem on singularities of differential equations}

\subjclass[2010]{37 F 50, 37 F 25.}

\author[R. P\'{e}rez-Marco]{Ricardo P\'{e}rez-Marco}
\address{CNRS, IMJ-PRG, Paris 7, Bo\^\i te courrier 7012, 75005 Paris Cedex 13, France}
\email{ricardo.perez.marco@gmail.com}

\thanks{.}

\begin{document}

\begin{abstract}
We solve Briot and Bouquet problem on the existence of non-monodromic (multivalued) solutions
for singularities of differential equations in the complex domain. The solution is an application of hedgehog 
dynamics for indifferent irrational fixed points. We present an important simplification 
by only using a local hedgehog for which we give a simpler and  direct construction 
of quasi-invariant curves which does not rely on complex renormalization.
\end{abstract}

\maketitle

\section{Introduction.}

We prove the following Theorem:

\begin{theorem}
Let $f(z)=e^{2\pi i \a} z +\cO(z^2)$, $\a \in \RR-\QQ$ be a germ of holomorphic diffeomorphism with 
an indifferent irrational fixed point at $0$.

There is no orbit of $f$ distinct from the fixed point at $0$ that converges to $0$ by positive or negative 
iteration by $f$.
\end{theorem}

This Theorem solves the question of C. Briot and J.-C. Bouquet on singularities of 
differential equations from 1856 (\cite{BB}), as well as questions of H. Dulac (1904, \cite{D1}, \cite{D2}), 
\'E. Picard (1896, \cite{P}), P. Fatou (1919, \cite{F}),
and two more recent conjectures of M. Lyubich (1986, \cite{Lyu}).

\medskip

The Theorem is trivial when the fixed point is linearizable, so, for the rest of the article, we 
assume that $f$ is not linearizable. 

\medskip

The main difficulty is to understand the non-linearizable dynamics. 
The proof relies on hedgehogs and their dynamics discovered by the author in \cite{PM5}. More precisely, 
we have from \cite{PM5} the existence of hedgehogs: 

\begin{theorem}[Existence of hedgehogs] \label{thm_hedgehogs}
 Let $U$ be a Jordan neighborhood of $0$ such that $f$ and $f^{-1}$ are defined and univalent on $U$, and continuous 
 on $\bar U$.
 
 There exists a hedgehog $K$ with the following properties:
 
 \begin{itemize}
  \item $0\in K \subset \bar U$
  \item $K$ is a full, compact and connected set.
  \item $K\cap \partial U \not= \emptyset$.
  \item $f(K)=f^{-1}(K)=K$.
 \end{itemize}

 Moreover, $f$ acts continuously on the space of prime-ends of $\CC-K$ and defines an homeomorphism of the circle of prime-ends with
 rotation number $\a$.
\end{theorem}

\medskip

\begin{figure}[ht]
\centering
\resizebox{6cm}{!}{\includegraphics{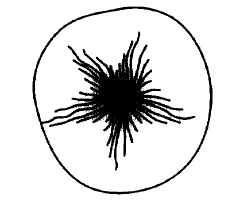}}    % name of the file - without extension
\caption{A hedgehog and its defining neighborhood.}
\end{figure}

\medskip

In the proof we only need to consider local hedgehogs, i.e. a hedgehog associated 
to a small disk $U=\DD_{r_0}$ with $r_0>0$
small enough. Let $K_0$ be the hedgehog associated to $\DD_{r_0}$. The two following Theorems imply our main 
Theorem.

\begin{theorem}\label{thm_unif}
Let $(p_n/q_n)_{n\geq 0}$ be the sequence of convergents of $\alpha$. We have
$$
\lim_{n\to +\infty} f_{/K_0}^{\pm q_n} ={\hbox {\rm id}}_{K_0} \ ,
$$
where the convergence is uniform on $K_0$.
\end{theorem}

Therefore all points of the hedgehog are uniformly recurrent, and no point 
on the hedgehog distinct from $0$ converges to $0$ by positive or negative iteration by $f$. 

\begin{theorem} \label{thm_oscullating}
Let $z_0 \in U-K_0$ such that the positive, resp. negative, orbit $(f^n(z_0))_{n\geq 0}$, resp. 
$(f^{-n}(z_0))_{n\geq 0}$, accumulates a point on $K_0$. Then this 
orbit accumulates all $K_0$,
$$
K_0\subset {\overline{(f^{n}(z_0))_{n\geq 0}}} \ \  ({\hbox{\rm {resp.}}} K_0\subset {\overline{(f^{-n}(z_0))_{n\geq 0}}} ) \ .
$$
\end{theorem}

In particular this implies that if such an orbit $(f^{n}(z_0))_{n\geq 0}$ (resp. $(f^{-n}(z_0))_{n\geq 0}$) accumulates $0\in K_0$ then it cannot 
converge to $0$. Note that if $f$ is not linearizable then it is clear that $0\in \partial K_0$.
Indeed one can prove that the hedgehog $K_0$ has empty interior and $K_0=\partial K_0$, but we don't need to use this fact. We can just prove the previous 
Theorem for $\partial K_0$.

\medskip

The proof of these two Theorems are done by constructing quasi-invariant curves near the hedgehog. These 
are Jordan curves surrounding the hedgehog and almost invariant by high iterates of the dynamics. The 
quasi-invariance property is obtained for the Poincar\'e metric of the complement of the hedgehog in the 
Riemann sphere.

\medskip

Therefore,  it is enough to carry out  the construction for local hedgehogs, and for these 
we have a direct and simpler construction of quasi-invariant curves, that does not rely on complex 
renormalization techniques. Classical one real dimensional estimates for smooth circle diffeomorphism 
combined with an hyperbolic version of Denjoy-Yoccoz Lemma in order to control the complex orbits for analytic circle diffeomorphisms, are enough. 
This gives an important simplification for local hedgehogs 
of the proof of the main Theorem that was announced in \cite{PM2}.

%\newpage

\section{Historical introduction on Briot and Bouquet problem.}

In 1856 C. Briot and J.-C. Bouquet published a foundational article \cite{BB} on the local solutions
of differential equations in the complex domain. They are particularly interested in how a local 
solution determines uniquely the holomorphic function through analytic continuation. 
They consider a first order differential equation of a differential equation 
of the form
$$
\frac{d y}{dx} = f(x,y) \ ,
$$
where $f$ is a meromorphic function of the two complex variables $(x,y)\in \CC^2$ in a neighborhood of a point 
$(x_0,y_0)$. A. Cauchy proved his fundamental Theorem on existence and uniqueness of local 
solutions\footnote{What is called today in Calculus books Cauchy-Lipschitz Theorem.}: 
If $f$ is finite and holomorphic in a
neighborhood of $(x_0,y_0)$ then there exists a unique holomorphic local solution $y(x)$ satisfying the 
initial conditions
$$
y(x_0) = y_0 \ .
$$

In their terminology, Briot and Bouquet talk about ``solutions monog\`enes et monodromes'', ``monog\`ene'' or monogenic
meaning $\CC$-differentiable, i.e. holomorphic, and ``monodrome'' or monodromic meaning univalued, 
since they also consider multivalued solutions with non-trivial monodromy at $x_0\in \CC$.

\medskip

Briot and Bouquet start their article by giving a simple proof of Cauchy Theorem by the majorant series method.
Then they consider the situation where $f$ is infinite or has a singularity at $(x_0, y_0)$. They observe that even 
in Cauchy's situation, we may get to such a point by a global analytic continuation of any solution. We 
assume for now on that $(x_0,y_0)=(0,0)$. Writing down $f$ as the quotient of two holomorphic germs
$$
f(x,y)=\frac{A(x,y)}{B(x,y)} \ ,
$$
they study the situation when $A(0,0)=B(0,0)=0$ (they call these singularities ``of the form $\frac{0}{0}$''). 
This is done in Chapter III, starting in section 75 of \cite{BB}.
After a simple change of variables, the equation reduces to
$$
x\frac{dy}{dx}= a y+bx +\cO(2) \ ,
$$
and a discussion starts considering the different cases for different values of the coefficients $a,b \in \CC$. 
They prove the 
remarkable Theorem that if $a$ is not a positive integer, then there always exists a holomorphic solution $y(x)$
in a neighborhood of $0$ vanishing at $0$ (Theorem XXVIII in section 80 of \cite{BB}). They show that this 
holomorphic solution is the only monodromic one and in their proof of uniqueness (in section 81) 
the equation is put in the form
$$
x\frac{dy}{dx}=y (a+\cO(2)) \ .
$$
In this last form the holomorphic solution corresponds to $y=0$.

\medskip

After that they proceed to show that when the real part of $a$ is positive there are infinitely many 
non-monodromic solutions (section 82 in \cite{BB}), i.e. holomorphic solutions $y(x)$ that are 
multivalued around $0\in \CC$.

\medskip

They make the claim in section 85 in \cite{BB} that when the real part of $a$ is negative there are no other 
solutions, not even non-monodromic, other than the holomorphic solution found.  

\medskip

The proof of this statement contains a gap. Starting with the new form of the differential equation
$$
x\frac{dy}{dx}=y(a +\cO_y(1))+xy \varphi(x,y) \ ,
$$
they transform it into
$$
\frac{dy}{y} + (A+By+\ldots)\,  dy = a\, \frac{dt}{t} + \psi (x,y) \, dt \ ,
$$
where $A+By+\ldots$ is a holomorphic function of $y$ near $0$ and $\psi$ is holomorphic near $(0,0)$.
Assuming by contradiction the existence of another solution, 
integration of the equation over a path from $x_1$ to $x$, $y_1=y(x_1)$, gives
$$
\log \left (\frac{y}{y_1} \right )+ (A(y-y_1)+\ldots )=\log \left ( \frac{x}{x_1}\right )^a 
+\int_{x_1}^x \psi(x, y(x)) \, dx \ .
$$
They pretend that this is of the form
$$
\log \left (\frac{y(x)}{y_1} \right )=\log \left ( \frac{x}{x_1}\right )^a 
+\eps,
$$
where $\eps$ is a small quantity, vanishing for $x=x_1$, and very small when $x\to 0$, to get 
a contradiction using that for $\Re a<0$,  $\Re \log (x/x_1)^a\to +\infty$ when $x\to 0$ 
but $\Re \log (y/y_1) \to -\infty$ if $y(x)\to 0$.

\medskip

Unfortunately $\eps$ is not small because since $y(x)$ is not monodromic, the integral 
$$
\int_{x_1}^x \psi(x, y(x)) \, dx
$$
is not monodromic either, and if the path of integration spirals around $0$ it can get arbitrarily large. 

\medskip

\'E. Picard observes (\cite{P} Vol. II p.314 and p.317, 1893, see also Vol. III p.27 and 29, 1896) 
that with some implicit assumptions (that are not in  \cite{BB}) 
the argument is correct if we approach $x=0$ along a path of finite 
length where the argument of $y(x)$ stays bounded or with a tangent at $0$, trying (not very convincingly) to rebate 
L. Fuchs that pointed out the error in \cite{Fu}. 
H. Poincar\'e does not mention the error in his article \cite{Po1} where 
he states Briot and Bouquet result without any restriction, 
and in his Thesis \cite{Po2} where he studies the case  where the real
part of $a$ is positive (and carefully avoids discussing further the 
other problematic case). 

\medskip

Picard, in his first edition of his ``Trait\'e d'Analyse'' (\cite{P}, Vol. III, page 30, 1896), 
casts no doubt about the correction of Briot and Bouquet statement:

\medskip

\textit{``Il resterait \`a d\'emontrer que ces deux int\'egrales sont, en dehors de toute hypoth\`ese, les seules qui
passent par l'origine ou qui s'en rapprochent ind\'efiniment. Je dois avouer que je ne poss\`ede pas une d\'emonstration 
rigoureuse de cette proposition, qui ne para\^\i t cependant pas douteuse.''}\footnote{\textit{``It remains to prove that these two solutions are,
without any assumption, the only ones passing through the origin or accumulating it. I have to admit that 
I don't have a proof of this fact but it doesn't seem doubtful.''}}

\medskip

He refers to the two Briot-Bouquet holomorphic solutions $y(x)$ and $x(y)$.
His belief is probably reinforced by the saddle picture for real solutions 
that clearly only exhibit two real solutions in $\RR^2$ passing through the singularity.

\medskip

A major progress came with the Thesis of H. Dulac published in 1904 in the Journal of the 
\'Ecole Polytechnique \cite{D1}. He proves the existence of an infinite number of distinct non-monodromic 
solutions when $a$ is a negative rational number, thus proving than Briot and Bouquet original claim is 
always false in the rational situation. From the introduction of \cite{D1} we can read

\medskip

\textit {``$\ldots $ on sait depuis bien longtemps, qu'il n'existe que deux 
courbes int\'egrales r\'eelles passant par l'origine. 
En est-il de m\^eme dans le champ complexe ? C'est une question 
qui restait en suspens et que les g\'eom\`etres penchaient 
\`a trancher par l'affirmative (Picard, Trait\'e d'Analyse, 
II (sic)\footnote{Volume III is the correct reference.}, p. 30). 
Or je prouve, au contraire, tout au moins dans le 
cas o\`u $\alpha$ est rationnel, qu'il existe une infinit\'e 
d'int\'egrales $y(x)$ s'annulant avec $x$ ($x$ tendant vers 
z\'ero suivant une loi convenable) $\ldots $''}
\footnote{``$\ldots $ from long time ago we know that there are only two real solutions 
passing through the origin. Is it the same in the complex? This is a question that remained open 
and that the geometers were inclined to decide in the affirmative (Picard, Trait\'e d'Analyse, 
II (sic), p. 30). But, on the contrary, I prove, at least in the case when $\a$ is rational, that 
there are infinitely many solutions $y(x)$ vanishing with $x$ ($x$ converging to $0$ under a suitable 
law) $\ldots$''}

\medskip

After Dulac's result Picard changed the quoted text in later editions of
his Trait\'e d'Analyse (\cite{P}, Vol. III, 3rd edition, p.30, 1928) into:

\medskip

\textit{``On a longtemps pr\'esum\'e que ces int\'egrales sont, en dehors de 
toute hypoth\`ese, les seules qui passent par l'origine
ou qui s'en rapprochent ind\'efiniment. 
Dans un excellent travail sur les points singuliers des \'equations 
diff\'erentielles, M. Dulac a d\'emontr\'e que la question 
\'etait tr\`es complexe. Prenons, par exemple, l'\'equation
$$
x\frac{dy}{dx}+y(\nu +\ldots )=0 \ ,
$$
o\`u $\nu$ est positif, \'equation \`a laquelle peut toujours se ramener le cas o\`u $\lambda$ est n\'egatif.
M. Dulac examine particuli\`erement le cas o\`u $\nu$ est un nombre rationnel $p/q$, et 
montre qu'il y a alors, en g\'en\'eral, 
une infinit\'e d'int\'egrales pour lesquelles $x$ et $y$ tendent vers $0$.''
}
\footnote{``For a long time it was believed that, without any further condition, these are the only solutions 
passing through or accumulating the origin.

In an excellent work on the singular points of differential equations, M. Dulac has proved that the question is very complex. Take 
for instance the equation
$$
x\frac{dy}{dx}+y(\nu +\ldots )=0 \ ,
$$
where $\nu$ is positive, equation that we can always reduce the case where $\lambda$ is negative.

M. Dulac examines specially the case where $\nu$ is a rational number $p/q$, and proves that in general there are an infinite 
number of solutions for which $x$ and $y$ converge to $0$.''}

\medskip

Dulac insisted in his Thesis that he had no answer for the irrational case (\cite{D1} p.4):

\medskip
\textit{``1. $\nu$ est irrationnel. On a un col. $H(x,y)$ existe formellement, mais est divergent, au moins dans certains cas. 
S'il y a des int\'egrales pour lesquelles $x$ et $y$ tendent simultan\'ement vers $0$, et si l'on d\'esigne par $\omega$ et $\theta$
les arguments de $x$ et $y$, quels que soient $m$ et $n$, $|x^my^n\omega|$ et $|x^m y^n\theta |$ croissent ind\'efiniment. Je ne puis 
me prononcer sur l'existence de ces int\'egrales.''}\footnote{``1.$\nu$ is irrational. We have a saddle. $H(x,y)$ exists formally, but is divergent,
at least in certain cases. If there are solutions $x$ and $y$ which tend simultaneously to $0$, and if we note $\omega$ and $\theta$ the arguments 
of $x$ and $y$, then for all $m$ and $n$, $|x^my^n\omega|$ and $|x^m y^n\theta |$ must grow indefinitely. I cannot decide on the 
existence of such solutions.'' }

The expression $yx^\nu H(x,y)$ is a formal first integral of the solutions and he discuss its convergence in p.20.
It is well known to Dulac that convergence of $H$ solves the problem.

\medskip

Then $30$ years later he  recalls that the problem remains unsolved (\cite{D2} p.31): 

\medskip

\textit{``Dans le cas $2$ ($\nu$ irrationnel, $h(x,y)$ divergent), on ne sait s'il existe des solutions nulles autres que  $x=0$, $y=0$. 
Ce sont l\`a deux questions qu'il y aurait grand int\'er\^et \`a \'elucider.''}\footnote{``In case $2$ ($\nu$ irrational, $h(x,y)$ divergent),
we don't know if there are null solutions other than $x=0$, $y=0$.''}

\medskip

Many results obtained by these distinguished geometers of the XIXth century where 
rediscovered in modern times, sometimes with a different point of view or language. The 
original problem of singularities of differential equations of the form $\frac{0}{0}$ (according to Briot and 
Bouquet terminology) is equivalent to study solutions of the holomorphic vector field $X=(B,A)$ near 
$(0,0) \in \CC^2$,

\begin{equation*}
\begin{cases}
\dot x &= B(x,y)  \\
\dot y &= A(x,y) 
\end{cases} 
\end{equation*}

The local geometry corresponds also to the 
study the holomorphic foliations on $\CC^2$ near the singular point $(0,0)$
defined by the differential form
$$
A(x,y) dx - B(x,y) dy =0 \ .
$$
The situation of Briot and Bouquet problem corresponds to an irreducible singularity with a  
non-degenerate linear part,
$$
(\a y+\cO(2)) dx + (x +\cO(2)) dy =0 \ ,
$$
where $-\a=a$ is Briot and Bouquet coefficient.

\medskip

When $\a\in \CC-\RR_+$, and $\a$ is neither a negative integer nor the inverse of a negative integer,  
we are in the Poincar\'e domain and the singularity is equivalent to 
the linear one. When $\a$ is a negative integer or its inverse, then we can conjugate the singularity to 
a finite Poincar\'e-Dulac normal form (see \cite{Ar} section 24). We assume $\a$ real and positive $\a >0$, which 
defines, in modern terminology, a singularity in the Siegel domain. 
The singularity is formally linearizable, but the convergence
of the linearization presents problems of Small Divisors. Precisely in this situation Dulac 
already proved in his Thesis 
the existence of non-linearizable singularities in section 12. This is a notable achievement that anticipates 
in several decades the non-linearization results for indifferent fixed points.
The existence of Briot and Bouquet holomorphic solution proves the existence of two leaves of the 
holomorphic foliation crossing transversally at $(0,0)$. This means that the singularity can be put 
into the form
$$
\a y (1+\cO(2)) dx + x (1+\cO(2)) dy =0 \ .
$$
Again $y=0$ corresponds to the Briot and Bouquet holomorphic solution. It is now easy to make the link 
with the original Briot and Bouquet $\frac{0}{0}$ singulatities of differential equations. Each solution $y(x)$, 
distinct from the only monodromic solution $y(x)=0$, with initial data $(x_0,y_0)$ close to $(0,0)$, has a graph 
over the $x$-axes that
corresponds to the leaf of the foliation passing through the point $(x_0,y_0)$. The multivaluedness or 
non-monodromic character
of the solution can be seen in the intersection of that leaf with a transversal $\{x=x_0\}$. The $y$-coordinates
of these points of intersection give the different values taken by the non-monodromic solution that are 
obtained by following a path in the leaf that projects onto the $x$-axes into a path circling around $x=0$.

\medskip

The topology of the foliation is understood through a holonomy construction (see \cite{MM}, and for the 
rational case see \cite{C}): 
Taking a transversal $\{x=x_0\}$
and lifting the circle $C(0, |x_0|) \subset \{y=0\}$ in nearby leaves, the return map following this 
lift in the negative orientation, defines a
germ of holomorphic diffeomorphism in one complex variable with a fixed point at $(x_0,0) \subset \{x=x_0\}$.
Taking a local chart in this complex line, we have a local holomorphic diffeomorphism 
$f\in {\hbox{\rm Diff}} (\CC , 0)$, $f(0)=0$, and linearizing the equations we can compute its linear part at $0$,
$$
f(z) =e^{2\pi i \a} z +\cO(z^2) \ .
$$
(to see this, note that $yx^{\a}$ is a first integral of the linearized differential form, thus is invariant of the solutions
in the first order)
Thus we get a germ of holomorphic diffeomorphism with an indifferent irrational fixed point.
It is obvious from the classical point of view that the local dynamics near $0$ of 
this return map contains the information about 
the non-monodromic solutions starting at $x=x_0$. Thus we transform our original problem into a problem 
of holomorphic dynamics. Note that we can also reconstruct all the foliation and a neighborhood 
of $(0,0)\in \CC^2$ minus 
the leave $\{y=0\}$ by continuing the continuing the complex leaves from the transversal. J.-F. Mattei and P. Moussu
proved in \cite{MM} that two singularities in the Siegel domain 
with conjugated holonomies are indeed conjugated in $\CC^2$ by 
``pushing'' the conjugacy along these leaves and using Riemann removability 
Theorem in $\CC^2$. J.-Ch. Yoccoz and 
the author proved in \cite{PMY} that the set of dynamical 
conjugacy classes of holonomies is in bijection with the 
set of conjugacy classes of singularities in the Siegel domain. The rational case was previously treated  by 
J. Martinet and J.-P. Ramis (\cite{MR1}, \cite{MR2}) by identifying the conjugacy invariants. 
This establish a full dictionary of the two problems. 
In particular, an interesting corollary is that Brjuno diophantine condition is optimal for analytic linearization of 
the singularity.

\medskip

For our problem, the existence of non-monodromic solutions vanishing with $x$ when $x\to 0$ following an 
appropriate path is equivalent to finding a leave that accumulates the singularity $(0,0)$ but distinct 
from the Briot and Bouquet leaves $\{x=0\}$ and $\{y=0\}$ and a path $\gamma$ on this leave converging to $(0,0)$. 
This path $\gamma$ projects properly in the $\{y=0\}$ plane into a spiral around $(0,0)$ and converging to 
$(0,0)$. The path $\gamma$ is homothopic in the leave to a path above $C(0,|x_0|)$ such that the iterates 
of the return map converve to $(x_0,0)$.  Since $\pi_1(\CC^*) \approx \ZZ$, 
this gives an orbit of the return map that has a positive or negative orbit converging to the indifferent 
fixed point. Conversely, if we have such an orbit of the return map, we can push homothopicaly the path in the 
leave close to $\{x=0\}$ to make it converging to $(0,0)$ (just using continuity of the foliation).

\begin{proposition}
 When $\a \in \RR_+-\QQ$, Briot and Bouquet non-existence of non-monodromic solutions vanishing at $0$ is equivalent to 
 the existence of an orbit distinct from $0$ that converges to $0$ by iteration by the return map $f$ or $f^{-1}$.
\end{proposition}

Since linearizable dynamics don't have this property, we see that C-L. Siegel  
linearization theorem (\cite{S}, 1942) shows that Briot and Bouquet statement is 
true when $\a \in \RR_+-\QQ$ satisfies the
arithmetic linearization condition that was improved later by A.D. Brjuno (\cite{Br}) to the so called 
Brjuno's condition
$$
\sum_{n=0}^{+\infty } \frac{\log q_{n+1}}{q_n} <+\infty \ .
$$
The sequence of $(p_n/q_n)_{n\geq 0}$ are the convergents of $\a$. 
The positive answer to Briot and Bouquet question in the linearizable case, that corresponds to $H(x,y)$ 
being convergent in Dulac's notation, was already well known to Dulac in \cite{D1}.

\medskip

Indeed the non-existence of non-monodromic for singularities of differential equations 
were well understood in the linearizable case, since H. Poincar\'e \cite{Po1} 
because linearization is equivalent to the existence of a 
first integral of the system of the form (see \cite{D1} section [])
$$
I(x,y)=yx^\a H(x,y) \ .
$$

\medskip

Note also that to have non-monodromic solutions $y(x)$ that accumulate into (but not converge to) $0$ when $x\to 0$ is a simpler problem 
that is equivalent for the monodromy dynamics to have an orbit that accumulates $0$ by positive or negative iteration. This was solved in 
general in \cite{PM5} by the discovery of hedgehogs and the result that almost all points in the hedgehog for the harmonic measure have 
a dense orbit in the hedgehog.

\medskip

What remains to be elucidated for the Briot and Bouquet problem is the non-linearizable case, and more precisely the following problem:

\begin{problem}
 Let $\a \in \RR-\QQ$ and $f(z)=e^{2\pi i \a} z +\cO(z^2)$ be a germ of holomorphic diffeomorphism with an 
 indifferent fixed point at $0$. Does there exists $z_0 \not= 0$ such that 
 $$
 \lim_{n\to +\infty } f^n(z_0) =0 \ .
 $$
\end{problem}

P. Fatou was confronted to this problem in his pioneer 
study of the dynamics of rational functions (\cite{F}, 1919) without knowing the 
relation to Briot and Bouquet problem. 
About fixed points ("points doubles" in 
Fatou's terminology) of holomorphic germs, which are 
indifferent, irrational and non-linearizable, Fatou writes \cite{F} p.220-221:

\medskip

\textit{``Il reste \`a \'etudier les points doubles dont le multiplicateur est de la forme 
$e^{i\alpha}$, $\alpha$ \'etant un nombre r\'eel inconmensurable avec $\pi$. Nous ne savons que fort 
peu de choses sur ces points doubles, dont l'\'etude 
du point de vue qui nous occupe para\^\i t tr\`es difficile. ($\ldots$)
Existe-t-il alors des domaines dont les cons\'equents tendent vers le point double ?
Nous ne pouvons actuellement ni en donner d'exemple, 
ni prouver que la chose soit impossible $\ldots$''} 
\footnote{\textit{``It remains to study fixed points with a multiplier of the 
form $e^{i\alpha}$, $\alpha$ being a real number incommensurable with $\pi$. We know little about these fixed points, and 
their study from our point of view appears very hard. ($\ldots$) Are there any domain such that the positive iterates 
converge to the fixed point? We cannot give examples nor rule out this possibility.''}}

\medskip

Fatou's question is related to the question of the non-existence of wandering components of the Fatou set for rational functions. 
This was only proved in 1985 by D. Sullivan \cite{Su}. Note that we do indeed have domains (that are not Fatou components) converging 
by iteration to rational indifferent fixed point as the local analysis of the rational case shows.

% 
% \begin{question}[P. Fatou, 1919]
% \textit{$\ldots$ Then do there exist domains with positive iterates converging to the fixed point? 
% We cannot give examples nor rule out that possibility $\ldots$}
% \end{question}
% 
% \medskip

\medskip

The non-existence of domains converging to an indifferent irrational fixed point 
was also conjectured by M. Lyubich in \cite{Lyu} p.73 (Conjecture 1.2), apparently unaware of Fatou's question. 
Lyubich also conjectured (Conjecture 1.5 (a) \cite{Lyu} p.77) that for any indifferent irrational non-linearizable fixed point 
there is a critical orbit that converges to the fixed point.

\medskip

The author proved in \cite{PM5} the Moussu-Dulac 
Criterium : $f$ is not linearizable if and only if 
$f$ has an orbit accumulating the fixed point $0$.
We may think that this could give support to the existence 
of a converging orbit. The discovery of hedgehogs gave new tools for the understanding of the non-linearizable
dynamics. Indeed, hedgehogs are the central tool in the final solution of all this problems:

\begin{theorem}\label{thm_main}
There is no orbit converging by positive 
or negative iteration to an indifferent irrational fixed point 
of an holomorphic map and distinct from the fixed point.
\end{theorem}

Therefore, the Briot and Bouquet problem has a positive solution in the irrational case. The questions of Dulac, 
Picard, Fatou are solved. Lyubich's Conjecture 1.2 in \cite{Lyu} has a positive answer, but conjecture 1.5 (a) in \cite{Lyu}
is false: For a generic rational function, there is no critical point converging to an indifferent irrational non-linearizable 
periodic orbits. There may be pre-periodic critical points to this orbit, but this is clearly non-generic. We may formulate a 
proper conjecture that has better chances to hold true:

\begin{conjecture}
Let $f$ be a rational function of degree $2$ or more, with an indifferent irrational non-linearizable fixed point $z_0$. 
There exists a critical point $c_0$ of $f$, such that 
$$
\lim_{n\to +\infty} \frac1n \sum_{j=0}^{n-1} \delta_{f^j(c_0)} \to \delta_{z_0} \ .
$$
\end{conjecture}

\medskip

Theorem \ref{thm_main} was announced in \cite{PM2} and a complete proof was given in the unpublished manuscript \cite{PM6}.
The proof given here concentrates on this particular Theorem and the solution of Briot and Bouquet problem, and 
not the many other properties of general hedgehog's dynamics. The proof follows the same lines as in \cite{PM6}, but  
we have incorporated several new ideas that greatly improve and simplify the technical part of construction 
of quasi-invariant curves 
that are fundamental in the study of the hedgehog dynamics. It was recently noticed in \cite{PM7} an hyperbolic 
interpretation of Denjoy-Yoccoz Lemma that controls orbits of an analytic circle diffeomorphism $g$ in a complex 
neighborhood of the circle. Then, when we control the non-linearity $||D\log Dg||_{C^0}$ of $g$,
we can construct directly the quasi-invariant curves without complex renormalization.
The second observation if that in the proof of Theorem \ref{thm_main} we can work with local hedgehogs (small 
hedgehogs). Then the associated circle diffeomorphism has a small non-linearity and the construction of quasi-invariant curves is easier. 

\newpage

\section{Analytic circle diffeomorphisms.}

\subsection{Notations.}

We denote by $\TT =\RR/\ZZ$ the abstract circle, and $\SS^1 =
E (\TT )$ its embedding in the complex plane $\CC$ given by 
the exponential mapping $E(x)=e^{2\pi i x}$.

We study analytic diffeomorphisms of the circle, but we 
prefer to work at the level of the universal covering, the real line, with 
its standard embedding $\RR \subset \CC$. We denote by $D^\omega (\TT)$ 
the space of non decreasing analytic diffeomorphisms $g$ of the real line 
such that, for any $x\in \RR$, $g(x+1)=g(x)+1$, which is the commutation to the 
generator of the deck transformations $T(x)=x+1$. An element of the space 
$D^\omega (\TT)$ has a well defined rotation number $\rho (g)\in \RR$. The order preserving 
diffeomorphism $g$ is conjugated to the rigid translation $T_{\rho(g)}: x\mapsto x+\rho(g)$, 
by an orientation preserving homeomorphism $h:\RR \to \RR$, such that $h(x+1)=h(x)+1$.

For $\Delta >0$, we note $B_{\Delta} =\{ z\in \CC ; |\Im z | < \Delta \}$, and 
$A_\Delta = E(B_\Delta )$. The subspace $D^\omega (\TT, \Delta )\subset D^\omega (\TT )$
is composed by the elements of $D^\omega (\TT)$ which extend analytically to 
a holomorphic diffeomorphism, denoted again by $g$, such that $g$ and $g^{-1}$
are defined on a neighborhood of $\bar B_\Delta$.

\subsection{Real estimates.}

We refer to \cite{Yo3} for the results on this section. We assume that the orientation preserving 
circle diffeomorphism $g$ is $C^3$
and that the rotation number $\alpha =\rho (g)$ is irrational. 
We consider the convergents $(p_n /q_n)_{n\geq 0}$ of $\alpha$ obtained by the continued fraction algorithm (see \cite{HW} for 
notations and basic properties of continued fractions). 

\medskip

For $n\geq 0$, we define the map $g_n (x) =g^{q_n}(x)-p_n$ and the intervals $I_n(x)=[x, g_n(x)]$,
$J_n(x)=I_n(x) \cup I_n(g_n^{-1}(x)) =[g_n^{-1}(x), g_n(x)]$. Let 
$m_n (x)=g^{q_n} (x)-x-p_n=\pm |I_n(x)|$,  $M_n =\sup_{\RR} |m_n (x)|$, and $m_n =\min_{\RR} |m_n (x)|$. 
Topological linearization is equivalent to $\lim_{n\to +\infty } M_n =0$. This is always true 
for analytic diffeomorphisms by Denjoy's Theorem, that holds for $C^1$ diffeomorphisms
such that $\log Dg$ has bounded variation. 

\medskip

Since $g$ is topologically linearizable, combinatorics of the irrational translation (or the continued fration algorithm) shows:

\begin{lemma} \label{lem_comb}
Let $x\in \RR$, $0\leq j < q_{n+1}$ and $k\in \ZZ$ the intervals $g^j\circ T^k(I_n(x))$
have disjoint interiors, and the intervals $g^j\circ T^k(J_n(x))$ cover $\RR$ at most twice.
\end{lemma}

\medskip

We have the following estimates on the Schwarzian 
derivatives of the iterates of $f$, for $0\leq j\leq q_{n+1}$,
$$
\left |S g^j (x)\right | \leq \frac{M_n e^{2V}S}{|I_n(x)|^2} \ ,
$$
with $S=||Sg||_{C^0(\RR )}$ and $V=\hbox{\rm {Var}} \log Dg$.

This implies a control of the non-linearity of the iterates (Corollary 3.18 in \cite{Yo3}): 

\begin{proposition}
For $0\leq j\leq 2q_{n+1}$, $c=\sqrt{2S} e^V$,
$$
|| D \log Dg^j ||_{C^0(\RR )} \leq c \, \frac{M_n^{1/2}}{m_n} \ .
$$
\end{proposition}

These give estimates on $g_n$. More precisely we have (Corollary 3.20 in \cite{Yo3}):

\begin{proposition} \label{prop_estimate} For some constant $C >0$, we have
 $$
 ||\log Dg_n ||_{C^0(\RR )}\leq C M_n^{1/2} \ .
 $$
\end{proposition}

\begin{corollary}  \label{cor_1}
For any $\eps>0$, there exists $n_0\geq 1$ such that for $n\geq n_0$,  we have
$$
 ||Dg_n -1||_{C^0(\RR )}\leq \eps \ .
 $$
\end{corollary}

\begin{proof}
Take $n_0 \geq 1$ large enough so that for $n\geq n_0$, $C M_n^{1/2} <\min (\frac{2}{3}\eps, \frac12) $, then use Proposition \ref{prop_estimate} and 
$\left |e^w-1 \right | \leq \frac{3}{2} |w|$ for $|w|<1/2$.
\end{proof}

\begin{corollary}\label{cor_estimate}
For any $\eps>0$, there exists $n_0\geq 1$ such that for $n\geq n_0$, for any $x\in \RR$ and $y\in I_n(x)$ we have
$$
1-\eps \leq \frac{m_n(y)}{m_n(x)}\leq 1+\eps \ .
$$
\end{corollary}

\begin{proof}
 We have $D m_n(x)=Dg_n(x)-1$, and 
 $$
 \left | m_n(y)-m_n(x)\right | \leq ||Dm_n ||_{C^0(\RR )} |y-x| \leq ||Dg_n -1||_{C^0(\RR )} |m_n(x)| \ .
 $$
 We conclude using Lemma \ref{cor_1}.
\end{proof}

\newpage

\section{Hyperbolic Denjoy-Yoccoz Lemma.}

With these real estimates for the iterates, and, more precisely, a control on the non-linearity, 
we can use them to control orbits in a complex neighborhood. 
We give here a version of Denjoy-Yoccoz lemma (Proposition 4.4 in \cite{Yo3})
that is convenient for our purposes.

\medskip

Given $\D >0$, we consider $g\in D^\omega (\TT , \D)$ such that $\inf_{B_\D} \Re Dg >0$ so that $\log Dg$ is a well 
defined univalued holomorphic function in $B_\D$. Given $g\in D^\omega (\TT )$ we get always 
this for a $\D >0$ small
enough (as in \cite{Yo3}), but here we don't need to make the assumption that for a given $g$,  $\D$ is small enough. 

\medskip

We do assume that we have a small non-linearity in $B_\D$, more precisely,
$$
\tau=|| D\log Dg ||_{C^0(B_\D)} <1/9 \ .
$$

\begin{lemma}\label{lem_DenjoyYoccoz}
Let $n_0 \geq 1$ large enough such that for all $n\geq n_0$, $M_n <\D/2$.

For $x_0\in \RR$, let $0<y_0 \leq 1$ and 
$$
z_0= x_0+im_n(x_0) y_0 \ .
$$
Then for $0\leq j\leq q_{n+1}$, $y_j\in \CC$, $\Re y_j>0$, is well defined by
$$
z_j=g^j(z_0)=g^j(x_0) + i m_n(g^j(x_0)) y_j \ ,
$$
and we have
$$
|y_j-y_0|\leq \frac{3}{4} \, y_0 \ .
$$
\end{lemma}

\begin{proof}
For $0<t\leq 1$ we define more generally
$$
z_{0,t}= x_0+im_n(x_0) t y_0 \ ,
$$
and we prove that $y_{j,t}\in \CC$, $\Re y_{j,t}>0$, is well defined by
$$
z_{j,t}=g^j(z_{0,t})=g^j(x_0) + i m_n(g^j(x_0)) y_{j,t} \ ,
$$
and that we have
$$
|y_{j,t}-y_{0,t}|\leq \frac{3}{4} \, y_{0,t} \ .
$$
Note that this last inequality implies $\Re y_{j,t} \leq \frac{7}{4} y_{0,t}$. 
The lemma corresponds to the case $t=1$.

We prove this result by induction on $0\leq j < q_{n+1}$ starting from $j=0$ for which the result is obvious.
Assuming it has been proved up to $0\leq j-1 < q_{n+1}$, then we have 
$$
0<\Im z_{j-1,t} \leq M_n \Re y_{j-1,t} \leq M_n \frac{7}{4} \, t y_0 < \frac{7}{8} \, \D <\D \ ,
$$
so $z_{j-1,t} \in B_\D$ and we can iterate once more and $z_{j,t} = g(z_{j-1,t})$ is well defined. We need to prove the estimate 
for $y_{j,t}$.
By the chain rule we have
$$
\log D g^j (z_{0,t}) =\sum_{l=0}^{j-1} \log D g (z_{l,t}) \ .
$$
Therefore, we have
\begin{align*}
\left | \log D g^j (z_{0,t}) -\log D g^j (x_0) 
\right | &\leq \sum_{l=0}^{j-1} 
\left | \log D g (z_{l,t}) -\log D g (x_l) \right | \\
&\leq \tau \sum_{l=0}^{j-1} |z_{l,t} -x_l| \\
&\leq \tau \sum_{l=0}^{j-1} |m_n(x_l)| |y_{l,t}| \\
&\leq  \frac{7}{4} \tau   ty_{0} \sum_{l=0}^{j-1} |m_n (x_l)| \\
&\leq  \frac{7}{4} \tau    \sum_{l=0}^{j-1} |m_n (x_l)|  \ .
\end{align*}

Considering the $j$-iterate of $g$ on the interval $I_n(x_0)$, we obtain a point 
$\zeta \in  ]x_0, g^{q_n} (x_0)-p_n[$ such that,
$$
Dg^j (\zeta ) =\frac {m_n (x_j)}{m_n (x_0)} \ ,
$$
and 
$$
\left | \log D g^j (\zeta ) -\log D g^j (x_0) 
\right | \leq \tau |m_n (x_0)|\leq \tau 
\sum_{l=0}^{j-1} |m_n (x_l)| \ .
$$
Adding the two previous inequalities, we have 
$$
\left | \log D g^j (z_{0,t}) -\log \frac {m_n (x_j)} 
{m_n (x_0)} \right | \leq \frac{11}{4} \tau 
\sum_{l=0}^{j-1} |m_n (x_l)| \ .
$$
The intervals $I_n(x_l)$, $0\leq l < q_{n+1}$,  being disjoint modulo $1$, we have 
$$
\sum_{l=0}^{q_{n+1}-1} |m_n (x_l)| < 1 \ .
$$
So we obtain
$$
\left | \log D g^j (z_{0,T}) -\log \frac {m_n (x_j)}{ 
m_n (x_0) } \right | \leq \frac{11}{4} \, \tau \ ,
$$
and taking the exponential (using $|e^w-1|\leq 3/2 |w|$, 
for $|w| <1/2$, since $\tau <1/9$ and $\frac{11}{4}\, \tau <\frac12$), we have
$$
\left | D g^j (z_{0,t})-\frac {m_n (x_j)}{m_n (x_0) } 
\right | \leq \frac{33}{8} \, \tau \frac{m_n (x_j)}{ m_n (x_0) }.
$$
Now, integrating along the vertical segment $[x_0, z_{0,t}]$ we get
$$
\left | g^j (z_{0,t}) -g^j (x_0) -i y_0  m_n (x_j) 
\right | \leq \frac{33}{8} \, \tau   y_{0,t} |m_n (x_j)| \ ,
$$
which, using $\tau <1/9$, finally gives
$$
| y_{j,t} - y_{0,t} | \leq \frac{11}{24} \, y_{0,t} <\frac34 \, y_{0,t} \ .
$$

\end{proof}

\subsection{Flow interpolation in $\RR$.} Since $g$ is analytic, from Denjoy's Theorem we know that $g_{/\RR}$ 
is topologically linearizable, i.e. there exists an non-decreasing 
homeomorphism $h: \RR \to \RR$, such that for $x\in \RR$,
$h(x+1)=h(x)+1$, and 
$$
h^{-1} \circ g \circ h =T_{\alpha} \ ,
$$
where $T_\alpha : \RR \to \RR$, $x\mapsto x+\alpha $.

We can embed $g$ into a topological flow on the real line $(\varphi_t)_{t\in \RR}$ defined for $t \in \RR$ by  
$\varphi_t =h\circ T_{t\alpha }\circ h^{-1}$. When $g$ is analytically linearizable the diffeomorphisms of 
this flow are analytic circle diffeomorphisms, but in general, when $g$ is not analytically
linearizable  the maps $\varphi_t$ are only  
homeomorphism of the real line, although for $t \in \ZZ + \alpha^{-1} \ZZ$, $\varphi_{t}$ is analytic since 
$\varphi_t$ is an iterate of $g$ composed by an integer translation. This can happen that for other values 
of $t$, where $\varphi_t$ can be an analytic diffeomorphism from the analytic centralizer of $g$ since
$\varphi_t\circ g =g\circ \varphi_t$. We refer to \cite{PM3} for more information on this fact and examples 
of uncountable analytic centralizers for non-analytically linearizable dynamics. 
Now $(\varphi_t)_{t\in [0,1]}$ is an 
isotopy from the identity to $g$. The flow $(\varphi_t)_{t\in \RR}$ is  a one 
parameter subgroup of homeomorphisms of the real line commuting to the 
translation by $1$.

\medskip

\subsection{Flow interpolation in $\CC$.} There are different complex extensions of the flow $(\varphi_t)_{t\in \RR}$
suitable for our purposes. For each $n\geq 0$, we can extend this topological 
flow to a topological flow $\cF_n$ in $\CC$ by defining, 
for  $z_0 =x_0 +i \, | m_n (x_0)| y_0 \in \CC$, with $x_0, y_0 
\in \RR$, 
$$
\varphi_{t}^{(n)} (z_0)=z_0(t)= \varphi_t(x_0) +i \, |m_n (\varphi_t(x_0))| y_0 \ .
$$
We denote $\Phi^{(n)}_{z_0}$ the flow line passing through $z_0$,
$$
\Phi^{(n)}_{z_0} = (\varphi_{t}^{(n)} (z_0))_{t\in \RR}.
$$

\medskip

\subsection{Hyperbolic Denjoy-Yoccoz Lemma.} We are now ready to give a 
geometric version of Denjoy-Yoccoz Lemma. We 
denote by $d_P$ the Poincar\'e distance in the upper half plane.

\begin{lemma}[Hyperbolic Denjoy-Yoccoz Lemma] \label{lem_DYhyp}
Let $\Delta >0$ and $g\in D^{\o} (\TT , \D )$ such that 
$$
||D \ \log \ Dg ||_{C^0 (B_{\D} )} < 1/9 \ .
$$
Let $n_0 \geq 1$ large enough such that for all $n\geq n_0$, $M_n <\D/2$.
% 
% There exists $\eps_0 >0$ small enough universal constant such that the following holds. Let $4<D_0<\frac{1}{4\eps_0}$.
% Let $\D >0$ and $g\in D^{\o} (\TT , \D )$ holomorphic 
% and continuous on $\overline {B_{\D}}$ such that $||D \ \log \ Dg ||_{C^0 (\overline {B_{\D}} )} <\epsilon_0$.
% Then there exists $n_0\geq 1$ such that for $n>n_0$,  we have 

Let $z_0= x_0+i |m_{n}(x_0)| y_0$, with 
$0<y_0<1$, so $z_0 \in B_{\D}$. Then for $0\leq j\leq q_{n+1}$ we have that the $(g^j(z_0))$ piece of orbit
follows at bounded distance the flow $\cF_n$ for the Poincar\'e metric of the upper half plane. More precisely we have
$$
d_P(g^j(z_0), \varphi^{(n)}_j(z_0)) \leq C_0 \ ,
$$
for some constant $C_0 >0$ (we can take $C_0=3$).
\end{lemma}

\begin{proof}
 We just use Lemma \ref{lem_DenjoyYoccoz} reminding that the Poincar\'e metric in the 
 upper half plane is given by $|ds| =\frac{|d\xi|}{\Im \xi}$ and
 \begin{align*} 
 d_P(z_j , \varphi_j^{(n)}(z_0)) &\leq \int_{[z_j , \varphi_j^{(n)}(z_0)]} \frac{|d\xi|}{\Im \xi}  \\
 &\leq |m_{n}(x_j)|\, . \,|y_j-y_0| 
 \, \frac{1}{\inf_{\xi \in [z_j , \varphi_j^{(n)}(z_0)]} \Im \xi} \\
 &\leq |m_{n}(x_j)|\, . \, |y_j-y_0| \, \frac{4}{|m_n(x_j)| \, y_0} \\
 &\leq 4\, \frac{|y_j-y_0|}{y_0} \leq 3=C_0
 \end{align*}
 where in the second inequality we used that $\Re y_j \geq \frac{1}{4} y_0$ which follow from $|y_j-y_0|\leq \frac{3}{4} y_0$
 that we also used in the last inequality.
\end{proof}

\newpage

\section{Quasi-invariant curves for local hedgehogs.}\label{sec_quasi}

Now we construct quasi-invariant curves for $g$ under the previous assumptions: $g\in D^\o(\TT, \D )$ and 
$$
\tau=|| D\log Dg ||_{C^0(B_\D)} <1/9 \ .
$$

\begin{theorem} [Quasi-invariant curves] \label{thm_quasi}
 Let $g$ be an analytic circle diffeomorphism with irrational rotation number $\a$. 
 Let $(p_n/q_n)_{n\geq 0}$ be the sequence of convergents of $\alpha$ given 
 by the continued fraction algorithm.
 
 Given $C_0>0$ there is  $n_0 \geq 0$ large enough such that there is a sequence of Jordan curves $(\gamma_n)_{n\geq n_0}$ for $g$ which are  
 homotopic to $\SS^1$ and exterior to $\overline{\DD}$ such that all the iterates
 $g^j$, $0\leq j\leq q_{n}$, are defined in a neighborhood of the closure of the annulus $U_n$ bounded by $\SS^1$ and $\gamma_n$, 
 and we have 
 $$
 \cD_{P}(g^j(\gamma_n), \gamma_n) \leq C_0 \ ,
 $$
 where $\cD_P$ denotes the Hausdorff distance between compact sets associated to $d_P$, the Poincar\'e distance in $\CC-\overline{\DD}$.
 We also have for any $z\in \gamma_n$, $d_P(g^{q_{n}} (z), z) \leq C_0$, that is,
 $$
 ||g^{q_{n}} -{\hbox{\rm id}}||_{C^O_P(\gamma_n)} \leq C_0 \ .
 $$
\end{theorem}

We choose the flow lines $\gamma_{n+1}=\Phi^{(n)}_{z_0}$, with $y_0 > 1/2$ and $n\geq n_0$ for $n_0\geq 1$ large enough, 
for the quasi-invariant curves of the Theorem. These flow lines are graphs over $\RR$. Given 
an interval $I\subset \RR$, we label $\tilde I^{(n)}$ the piece of $\Phi^{(n)}_{z_0}$ over $I$. 

\begin{lemma}\label{lem_bounded}
 There is $n_0\geq 1$  such that for $n\geq n_0$ and for  any $x\in \RR$, the piece $\tilde I_n^{(n)}(x)$  has bounded Poincar\'e diameter.
\end{lemma}

\begin{proof}
Let $z=x+i \, |m_n(x)| y_0$ be the current point in  $\tilde I_n^{(n)}(x)$. We have
$$
dz=\left ( 1 \pm i\, (Dg_n(x)-1) y_0 \right ) \, dx \ .
$$
For any $\eps_0 >0$,  choosing $n_0\geq 1$ large enough, for $n\geq n_0$, according to Lemma \ref{cor_1} we have
$$
\left |\frac{dz}{dx} -1 \right | \leq \eps_0 \ .
$$
Therefore, we have
$$
l_P (\tilde I_n^{(n)}(x_0))=\int_{\tilde I_n^{(n)}(x_0)} \frac{1}{|m_n(x)|\, y_0} \, |dz| \leq \int_{I_n(x_0)} \frac{1}{|m_n(x)|\, y_0} \, (1+\eps_0 )\, dx  \ .
$$
Now using Lemma \ref{cor_estimate} with $\eps= \eps_0$ and increasing $n_0$ if necessary, we have
$$
l_P (\tilde I_n^{(n)}(x))\leq \int_{I_n(x_0)} \frac{1}{|m_n(x_0)|\, y_0} \, \frac{1+\eps_0}{1-\eps_0} \, dx \leq \frac{1}{y_0} \frac{1+\eps_0}{1-\eps_0} 
\leq 2 \, \frac{1+\eps_0}{1-\eps_0}\leq C\ .
$$

\end{proof}

We assume $n\geq n_0$ from now on in this section and the next one.

\begin{lemma}\label{lem_cover}
 For $0\leq j < q_{n+1}$ and any $x\in \RR$, 
 the pieces $(g^j\circ T^k(\tilde J_n^{(n)}(x)))_{0\leq j\leq q_{n+1}, k\in \ZZ}$ have bounded Poincar\'e 
 diameter and cover $\Phi^{(n)}_{z_0}$.
\end{lemma}

\begin{proof}
 From Lemma \ref{lem_bounded} any $\tilde I^{(n)}_n(x)$ has bounded Poincar\'e diameter, thus also any
 $\tilde J^{(n)}_n(x)= \tilde I^{(n)}_n(x) \cup \tilde I^{(n)}_n(g_n^{-1}(x))$. Moreover, 
 we have $g^j\circ T^k(J_n(x))=J_n(g^j\circ T^k(x))$, and all $\tilde J^{(n)}_n(g^j\circ T^k(x)) $ 
 have also bounded Poincar\'e diameter. 
 From Lemma \ref{lem_comb} these pieces cover $\Phi^{(n)}_{z_0}$.
\end{proof}

\begin{corollary}\label{cor_dense}
 For some $C_0>0$, the flow orbit $(\varphi_{j,k}^{(n)} (z_0))_{0\leq j< q_{n+1}, k\in \ZZ}$ 
 is $C_0$-dense in $\Phi^{(n)}_{z_0}$ for the Poincar\'e metric.
\end{corollary}

We prove the first property stated in Theorem \ref{thm_quasi}:

\begin{proposition} Let $\gamma_n =\Phi^{(n-1)}_{z_0}$  for some $z_0$ from the previous lemma, then we have, for $0\leq j\leq q_n$,
$$
\cD_P(g^j(\gamma_n), \gamma_n)\leq 2C_0
$$
\end{proposition}

\begin{proof} We prove this Proposition for $n+1$ instead of $n$ (the proposition is stated to match $n$ in Theorem \ref{thm_quasi}).
It follows from the hyperbolic Denjoy-Yoccoz Lemma that the orbit $(g^j\circ T^k (z_0))_{0\leq j<q_{n+1}, k\in \ZZ}$ is $C_0$-close to flow orbit 
$(\varphi_{j,k}^{(n)} (z_0))_{0\leq j<q_{n+1}, k\in \ZZ}$, and from Corollary \ref{cor_dense} we have that a $2C_0$-neighborhood of $g^j(\gamma_{n+1})$ 
contains $\gamma_{n+1}$. Conversely, since we can chooose any $z_0 \in \gamma_{n+1}$, we also have that $g^j(\gamma_{n+1})$ is in a $C_0$-neighborhood
of $\gamma_{n+1}$.
\end{proof}

We prove the second property of Theorem \ref{thm_quasi}. We observe that $g^{q_{n+1}}(z_0) \in \tilde J_n^{(n)}(x_0)$, that $z_0 \in \tilde J_n^{(n)}(x_0)$,
and that $\tilde J_n^{(n)}(x_0)$ has a bounded Poincar\'e diameter by Lemma \ref{lem_cover}. Thus we get (taking a larger $C_0 >0$ if necessary):

\begin{proposition} For any $z_0\in  \Phi^{(n)}$ , we have 
$$
 d_P(z_0, g^{q_{n+1}}(z_0)) \leq C_0 \ .
$$
 \end{proposition}

%\newpage
 
\section{Osculating orbit.}\label{sec_osc}

 We prove the existence of an osculating orbit.

\begin{theorem} [Oscullating orbit]\label{thm_osc}
With the above hypothesis, for $n\geq n_o$ there exists a quasi-invariant curves $\gamma_n=\Phi^{(n-1)}_{z_0}$
such that the orbit $(g^j(z_0))_{0\leq j\leq q_n}$ is such that the union of Poincar\'e balls
$$
U_n=\bigcup_{0\leq j< q_{n}, k\in \ZZ} B_P(g^j(z_0)+k, C_0) \ ,
$$
separates $\RR$ from $\{\Im z > H\}$ with $H>0$ large enough, and any orbit $(g^j(w_0))_{j\in \ZZ}$ with $\Im w_0 >H$ with 
an iterate between $\gamma_n$ and $\RR$ has, for any $0\leq j\leq q_n$, an iterate in 
$$
\bigcup_{k\in \ZZ} B_P(g^j(z_0)+k, C_0) \ .
$$
\end{theorem}

%  
% \begin{theorem} [Transient annulus]
% With the above hypothesis, for $n\geq n_o$ there exists a transient annulus $A_n$ bounded by two quasi-invariant curves $\gamma^{+}_n$
% and $\gamma_n^{-}$ and an orbit $(g^j(z_0))_{0\leq j\leq q_n} \subset A_n$ such that $(g^j(z_0)+k)_{0\leq j\leq q_n, k\in \ZZ}$ is 
% $C_0$-dense in $A_n$ and such that any orbit above $A_n$ that has an iterate between $\RR$ and $A_n$ has, for any $0\leq j\leq q_n$, an iterate that is $C_0$-close 
% to a point $(g^j(z_0)+k$, for some $k\in \ZZ$.
% \end{theorem}
% 

From Lemma \ref{lem_cover} we get the property that the hyperbolic balls 
$B_P(\varphi^{(n)}_{t+k}(z_0), C_0)$ cover $\Phi^{(n)}_{z_0}$.

\begin{lemma} \label{lem_cov} We have that 
$$
U_n=\bigcup_{0\leq j< q_{n+1}, k\in \ZZ} B_P(\varphi^{(n)}_{t+k}(z_0), C_0)
$$
is a neighborhood of the flow line $\Phi^{(n)}_{z_0}$
\end{lemma}

\begin{proof}
We prove Theorem \ref{thm_osc}. In the following argument $C_0$ will denote several universal constants.
Enlarging the constant $C_0$, and using Lemma \ref{lem_DYhyp} we can replace the points $\varphi^{(n)}_{t+k}(z_0)$ by the points $g^j(z_0)+k$ in the orbit 
of $z_0$ in Lemma \ref{lem_cov}. Also, any orbit that jumps over $\gamma_n$ (by positive or negative iteration) as in Theorem \ref{thm_osc} has to visit
a $C_0$-neighborhood of $\gamma_n$ , and will be $C_0$-close to a point $z_1\in \gamma_n$ and then will be $C_0$-close to the $q_n$-orbit 
of $z_1$ modulo $1$. Finally we can replace $z_1$ by $z_0$ using that each point of the $q_n$-orbit of $z_1$ is $C_0$-close to a point in the $q_n$-orbit 
of $z_0$ modulo $1$ (enlarge $C_0$ if need be). 
\end{proof}

% 
% 
% construct a ``transient annulus'', i.e. any orbit from the outside of $\gamma_n$ that has iterates in between $\gamma_n$ and the circle 
% $\SS^1$ must visit $C_0$-close any point of $\gamma_n$.
% 
% \begin{proposition} Let $(g^j(z))_{0\leq j\leq q_{n+1}}$ be an orbit that starts on a point $z$ exterior to $\gamma_n$ and has some iterate in 
% between $\SS^1$ and $\gamma_n$. Then for any $w\in \gamma_n$, there is an iterate $g^j(z)$ such that
% $$
% d_P(g^j(z),w) \leq C_0 \ .
% $$
% \end{proposition}

%\newpage

\section{Proof of the main Theorem.}

We prove Theorems \ref{thm_unif} and \ref{thm_oscullating} that imply the main Theorem. We prove first the 
following preliminary Lemma that will allow us to work only with local hedgehogs.

\begin{lemma} \label{lem_bigmodulus}
 Let $g_n\in D^\omega (\TT, \D_n )$ with $\rho (g_n)=\alpha$ and $\D_n\to +\infty$. Then $g_n\to R_\a$ uniformly 
 on compact sets of $\CC^*$ and 
$$
\lim_{n\to +\infty} ||D \log D g_{n}||_{C^0(\RR )} =0 \ .
$$
\end{lemma}

\begin{proof} Let $\tilde g_n$ be the associated circle diffeomorphism.
The sequence $(\tilde g_n)$ is a normal family in $\CC^*$ (bounded inside $\DD$, and outside is the reflection across the unit circle), 
and any accumulation point is not constant since the 
unit circle is in the image of all $g_n$. Then by Hurwitz theorem any limit is an automorphism
of $\CC^*$, that extends to $0$ by Riemann's theorem, and so gives an automorphism of the plane 
leaving the unit circle invariant. The rotation number on the circle depends continuously on 
$\tilde g_n$ and is constant equal to $\alpha$, therefore the only possible limit of the sequence $(g_n)$ is $R_\a$. 
Since $D \log D R_\a =0$ we get the last statement.
\end{proof}

\medskip

We consider now the hedgehog $K_0$ given by Theorem \ref{thm_hedgehogs} for the domain $U=\DD_{r_0}$, and we use 
the relation between hedgehogs and analytic circle diffeomorphisms  
presented in \cite{PM5} to construct a circle diffeomorphism $g_0$.

\medskip

\begin{figure}[h]
\centering
\resizebox{9cm}{!}{\includegraphics{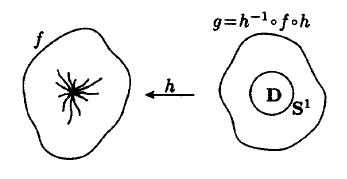}}    % name of the file - without extension
\caption{Relation between hedgehogs and circle maps.} 
\end{figure}

\medskip

We consider a conformal representation $h_0: \CC-\overline{\DD} \to \CC-K_0$ ($\DD$ is the unit disk), and we conjugate the dynamics to a univalent map 
$g_0$ in an annulus $V$ having the circle $\SS^1=\partial \DD$ as the inner boundary,
$$
g_0=h_0^{-1}\circ f\circ h_0 : V \to \CC \ .
$$
The topology of $K_0$ is complex (\cite{B1}, \cite{B2}, \cite{PM3}) and in particular $K_0$ 
is never locally connected, and $h_0$ does not extend to a continuous correspondence between $\SS^1$ and $\partial K_0$. 
Nevertheless, $f$ extends continuously to Caratheodory's prime-end compactification of $\CC-K_0$. This shows that $g_0$ 
extends continuously to $\SS^1$ and its Schwarz
reflection defines an analytic map of the circle defined on $V\cup \SS^1 \cup \bar V$, where $\bar V$ is the reflected annulus of $V$. Then 
it is not difficult to see that $g_0$ is an analytic circle diffeomorphism. We can also prove that $g_0$ has rotation number $\a$. This is 
harder to prove in general (for an aribtrary hedgehog), but it is not difficult to show that we can pick $K_0$ so that 
the rotation number of $g_0$ is $\a$ (see \cite{PM5} Lemma III.3.3) that is enough for our purposes. We choose such a $K_0$.
Therefore, the dynamics in a complex neighborhood of $K_0$ corresponds to the dynamics of an analytic circle diffeomorphism with 
rotation number $\a$.

\medskip

% 
% Let $g_{r_0} \in D^\omega(\RR)$ be the analytic circle diffeomorphism associated to the hedgehog $K_{r_0}$. 
% Let $h_{r_0}: \CC-\bar \DD \to \CC-K_{r_0}$ be the conformal representation such that 
% $$
% g_{r_0}=h_{r_0}^{-1}\circ f \circ h_{r_0} \ .
% $$

There is no risk of confusion and we denote also $g_0$ the lift to $\RR$.

\begin{theorem}
Let $\eps_0 >0$ and $\D >0$ be given

For $r_0>0$ small enough we have $g_{0} \in D^\omega(\TT, \D )$ and 
$$
||D \log D g_{0}||_{C^0(\RR )} <\eps_0 \ .
$$
\end{theorem}

\begin{proof}
 When $r_0\to 0$, we have $K_{0} \to \{ 0\}$ and the annulus where $g_{0}$ and $g_0^{-1}$ are defined has a modulus $M_0\to +\infty$.
Therefore, by Gr\"otsch extremal problem, for $r_0>0$ small enough we have $g_0 \in D^\omega(\TT, \D )$.  From Lemma \ref{lem_bigmodulus}
 $$
 \lim_{r_0\to 0} ||D \log D g_{0}||_{C^0(\RR )} =0 \ ,
 $$
 and the result follows.
\end{proof}

Let $\eps_0=1/9$ and $\D >0$ be as in Section \ref{sec_quasi} and Section \ref{sec_osc}. 
We fix now $r_0>0$ small enough such that $g_{0} \in D^\omega(\TT, \D )$, $\rho(g_0) =\a$, and
$$
||D \log D g_{0}||_{C^0(\RR )} <\eps_0 \ ,
$$
so that the hypothesis of Theorem \ref{thm_quasi} are fulfilled for $g_{0}$. Now we can apply Theorem \ref{thm_osc} and find 
a sequence $(\gamma_n)_{n\geq n_0}$ of quasi-invariant curves for $g_{0}$. We transport them by $h_{0}$ to get a
sequence of Jordan curves $(\eta_n)_{n\geq n_0}$
$$
\eta_n =h_{0} (\gamma_n) \ .
$$
We have 
$$
||g_{0}^{q_n}-{\hbox{\rm id}}||_{C_P^0(\gamma_n)} \leq C_0\ ,
$$
therefore, for the Poincar\'e metric of the exterior of the hedgehog,
$$
||f^{q_n}-{\hbox{\rm id}}||_{C_P^0(\eta_n)} \leq C_0\ ,
$$
and, since $\eta_n \to K_0$, for the euclidean metric, we have
$$
||f^{q_n}-{\hbox{\rm id}}||_{C^0(\eta_n)} =\eps_n \to 0\ .
$$
Thus, if $\Omega_n$ is the Jordan domain bounded by $\eta_n$, by the maximum 
principle we have
$$
||f^{q_n}-{\hbox{\rm id}}||_{C^0(\Omega_n)} =\eps_n \to 0\ .
$$
Since $\Omega_n$ is a neighborhood of $K_0$, $K_0\subset \bar \Omega_n$, we have 
$$
||f^{q_n}-{\hbox{\rm id}}||_{C^0(K_0)} =\eps_n \to 0\ .
$$
This proves Theorem \ref{thm_unif} for the positive iterates (same proof for the negative ones, or just apply the result to $f^{-1}$).

\medskip

We prove Theorem \ref{thm_oscullating} for $K_0$, or more precisely for $\partial K_0$ that was noted before that 
is enough for proving the Main Theorem (the hedgehog $K_0$ has empty interior 
and $K_0=\partial K_0$, but we don't need to use this fact).
For the proof of Theorem \ref{thm_oscullating} we transport by $h_{0}$ the Poincar\'e $C_0$-dense 
orbit $(g^j_{0}(z_0))_{0\leq j\leq q_n}$ given by Theorem \ref{thm_osc}. Let $\zeta_0=h_{0}(z_0)$ and 
$\cO_n=(f^j(\zeta_0))_{0\leq j\leq q_n}$ be this orbit.  Since $\eta_n \to \partial K_0$, we have, 
for $\eps_n\to 0$
$$
\cD(\cO_n,\partial K_0) \leq  \cD (\cO_n , \eta_n)+ \cD(\eta_n ,  \partial K_0) \leq \eps_n \ ,
$$
where $\cD$ denotes  the Hausdorff distance for the euclidean metric. 

Then any orbit starting outside of $\eta_n$ with an 
iterate inside $\eta_n$ must visit $\eps_n$-close for the euclidean metric any point of $\partial K_0$ that is strictly larger 
than $\{0\}$.

%\newpage

\end{document}